\documentclass[reqno]{amsart}

\usepackage{amsmath}
\usepackage{amsthm}
\usepackage{amsfonts}
\usepackage{amssymb}
\usepackage{enumerate}
\usepackage{todonotes}

\usepackage{xcolor}

\newcommand{\Le}{\mathcal{L}}

\DeclareMathOperator{\supp}{supp}

\newcommand{\Prob}{\mathbb{P}}
\newcommand{\E}{\mathbb{E}}

\renewcommand\Re{\operatorname{Re}}

\newcommand{\eps}{\varepsilon}

\renewcommand{\d}{\, d }

\theoremstyle{plain}
  \newtheorem{theorem}{Theorem}[section]
  
  \newtheorem{lemma}[theorem]{Lemma}

  \newtheorem{proposition}[theorem]{Proposition}

\theoremstyle{definition}
  \newtheorem{definition}[theorem]{Definition}
  \newtheorem{example}[theorem]{Example}
  \newtheorem{assumption}[theorem]{Assumption}

  \newtheorem{question}[theorem]{Question}
  
\theoremstyle{remark}
  \newtheorem{remark}[theorem]{Remark}

\begin{document}

\title{Sums of random polynomials with independent roots}

\author{Sean O'Rourke}
\address{Department of Mathematics\\ University of Colorado\\ Campus Box 395\\ Boulder, CO 80309-0395\\USA}
\email{sean.d.orourke@colorado.edu}
\thanks{S. O'Rourke has been supported in part by NSF grants ECCS-1610003 and DMS-1810500.}

\author{Tulasi Ram Reddy}
\address{New York University Abu Dhabi \\ United Arab Emirates}
\email{tulasi@nyu.edu}

\begin{abstract}
We consider the zeros of the sum of independent random polynomials as their degrees tend to infinity.  Namely, let $p$ and $q$ be two independent random polynomials of degree $n$, whose roots are chosen independently from the probability measures $\mu$ and $\nu$ in the complex plane, respectively.  We compute the limiting distribution for the zeros of the sum $p+q$ as $n$ tends to infinity.  The limiting distribution can be described by its logarithmic potential, which we show is the pointwise maximum of the logarithmic potentials of $\mu$ and $\nu$.  More generally, we consider the sum of $m$ independent degree $n$ random polynomials when $m$ is fixed and $n$ tends to infinity.  Our results can be viewed as describing a version of the free additive convolution from free probability theory for zeros of polynomials.  
\end{abstract}


\maketitle

\section{Introduction} 

Let $p$ and $q$ be monic polynomials in a single complex variable.  
%The purpose of this paper is to discuss the following natural question.  
This paper is concerned with the following natural question.  
\begin{question} \label{question:main}
Given the individual roots of $p$ and $q$, what are the roots of $p+q$?  
\end{question}
There has been considerable interest in the location of zeros of linear combinations of polynomials; we refer the reader to \cite{Fran,Fisk,SHK,SHK2,M,Pinter,Sodin,Ver,Walsh,Zedek} and references therein.  
%As an illustration, we point the reader to Theorem (17,1) in \cite{M}, which deals with the case when $p$ and $q$ have degrees $n_1$ and $n_2$, respectively, and the roots of $p$ and $q$ lie in the circular regions $C_1$ and $C_2$, respectively.  In this case, Theorem (17,1) in \cite{M} shows that the zeros of the sum $p+q$ lie in the locus of the roots of the equation 
%\[ (z - \alpha_1)^{n_1} + (z - \alpha_2)^{n_m} = 0 \]
%when $\alpha_1, \alpha_2$ vary over the regions $C_1, C_2$, respectively.  
The goal of this note is to address a probabilistic version of Question \ref{question:main}.  We focus on the model where $p$ and $q$ are independent random polynomials of the same degree, with roots chosen independently from the probability measures $\mu$ and $\nu$ in the complex plane, respectively.  Models of random polynomials with independent roots have been studied (mostly in relation to their critical points) in \cite{BLR,Hanin,Kab,OW2,OW,PR,RR,Sub} and references therein.  

\subsection{Main results}
Let $\mathcal{P}(\mathbb{C})$ be the set of probability measures on $\mathbb{C}$ with compact support.  The \emph{logarithmic potential} $U_\mu$ of $\mu \in \mathcal{P}(\mathbb{C})$ is the function $U_\mu: \mathbb{C} \to [-\infty, +\infty)$ defined for all $z \in \mathbb{C}$ by
\[ U_\mu(z) := \int_{\mathbb{C}} \log |z - w| \d \mu(w). \]

Let $\lambda$ denote Lebesgue measure on $\mathbb{C}$.  For a measure $\mu \in \mathcal{P}(\mathbb{C})$, we let $\supp(\mu) \subset \mathbb{C}$ denote the support of $\mu$.

\begin{definition} \label{def:circle}
Let $K \subset \mathbb{C}$ be nonempty.  We say $\mu \in \mathcal{P}(\mathbb{C})$ is \emph{supported on a circle centered in $K$} if there exist $r \geq 0$ and $z \in K$ so that 
\[ \supp(\mu) \subset \{w \in \mathbb{C} : |w - z| = r\}. \]
If not, we say \emph{$\mu$ is not supported on a circle centered in $K$}.  
If $\mu$ is supported on a circle centered in $\mathbb{C}$, we simply say \emph{$\mu$ is supported on a circle}, and when $\mu$ is not supported on a circle centered in $\mathbb{C}$, we say \emph{$\mu$ is not supported on a circle}.  
\end{definition}

For our main results, we will be interested in probability measures which are \emph{not} supported on circles.  For instance, if $\mu \in \mathcal{P}(\mathbb{C})$ is absolutely continuous with respect to the Lebesgue measure $\lambda$, then $\mu$ is not supported on a circle.  

Let $C^\infty_c(\mathbb{C})$ denote the set of all smooth functions $\varphi: \mathbb{C} \to \mathbb{C}$ with compact support.  Our main result is the following.

\begin{theorem} \label{thm:mainsimple2}
Let $\mu, \nu \in \mathcal{P}(\mathbb{C})$, and assume $\mu$ is not supported on a circle.  For each $n \geq 1$, define the degree $n$ polynomials 
\[ p_n(z) := \prod_{i=1}^n (z - X_i), \qquad q_n(z) := \prod_{i=1}^n (z - Y_i), \]
where $X_1, Y_1, X_2, Y_2, \ldots$ are independent random variables so that $X_i$ has distribution $\mu$ and $Y_i$ has distribution $\nu$ for each $i \geq 1$.  
Then there exists a (deterministic) probability measure $\rho$ on $\mathbb{C}$ so that, for any smooth and compactly supported function $\varphi: \mathbb{C} \to \mathbb{C}$, 
\[ \frac{1}{n} \sum_{i=1}^n \varphi(z_i^{(n)} ) \longrightarrow \int_{\mathbb{C}} \varphi \d \rho \]
in probability as $n \to \infty$, where $z_1^{(n)}, \ldots, z_n^{(n)}$ are the zeros of the sum $p_n + q_n$.  Here, $\rho$ depends only on $\mu$ and $\nu$ and is uniquely defined by the condition that 
\begin{equation} \label{eq:defrho}
	\int_{\mathbb{C}} \varphi \d \rho = \frac{1}{2 \pi} \int_{\mathbb{C}} \Delta \varphi(z) \left( \max \left\{ U_{\mu}(z), U_{\nu}(z) \right\} \right) d \lambda(z) \quad \text{ for all } \varphi \in C^\infty_c(\mathbb{C}).
\end{equation}  
\end{theorem}
  
\begin{remark} \label{rem:distribution}
Condition \eqref{eq:defrho}, which uniquely defines the measure $\rho$, can be succinctly written as
\[ \rho = \frac{1}{ 2 \pi} \Delta U \]
where $U(z) := \max \{ U_{\mu}(z), U_{\nu}(z) \}$ for each $z \in \mathbb{C}$ and where the Laplacian is interpreted in the distributional sense (see Section 3.7 in \cite{Ransford}).  
\end{remark}

\begin{remark}
The conclusion of Theorem \ref{thm:mainsimple2} fails to hold for some deterministic polynomials. For instance, consider the case when $p_n(z):=z^{n-\lfloor\frac{ n }{2}\rfloor}(z^{\lfloor\frac{ n }{2}\rfloor}-1)$ and $q_n(z):=z^{n-\lfloor\frac{ n }{2}\rfloor}(z^{\lfloor\frac{ n }{2}\rfloor}+1)$. The empirical distributions for the zeros of both $p_n$ and $q_n$ converge to a probability measure which is a mixture of the uniform measure on the unit circle and a point mass at the origin, where as the zeros of $p_n(z)+q_n(z)=2z^n$ are all located at the origin. 
\end{remark}

We emphasize that Theorem \ref{thm:mainsimple2} only requires $\mu$ to not be supported on a circle; $\nu$ may be an arbitrary compactly supported probability measure on $\mathbb{C}$.
This assumption on the support of $\mu$, however, is likely an artifact of our proof methods.  In Theorem \ref{thm:mainsimple} below, we give an alternative formulation that can be applied to measures supported on circles.  
For technical simplicity, we have focused on measures with compact support, but we anticipate that our main results should also hold for more general probability measures.

\begin{example} \label{example:disks}
Suppose $\mu$ is the uniform probability measure on the unit disk in the complex plane centered at $1$, and let $\nu$ be the uniform probability measure on the unit disk centered at $-1$.  Then for every $z \in \mathbb{C}$, 
\[ U_{\mu}(z) = \begin{cases} 
      \log |z-1|, & \text{ if } |z-1| > 1, \\
      \frac{1}{2} ( |z-1|^2 - 1), & \text{ if } |z-1| \leq 1,
   \end{cases} \]
and
\[ U_{\nu}(z) = \begin{cases} 
      \log |z+1|, & \text{ if } |z+1| > 1, \\
      \frac{1}{2} ( |z+1|^2 - 1), & \text{ if } |z+1| \leq 1,
   \end{cases} \]
see for instance \cite{ST}.  It follows that
\[ U(z) := \max\{ U_{\mu}(z), U_{\nu}(z) \} = \begin{cases} 
      \log |z-1|, & \text{ if } \Re(z) \leq 0, \\
      \log |z+1|, & \text{ if } \Re(z) > 0. 
   \end{cases} \]
In this case, an integration by parts argument shows that
\[ \frac{1}{2 \pi} \int_{\mathbb{C}} \Delta \varphi(z) U(z) d \lambda(z) = \frac{1}{\pi} \int_{\mathbb{R}} \varphi(iy) \frac{1}{1 + y^2} dy \]
for any $\varphi \in C_c^\infty(\mathbb{C})$.  Thus, \eqref{eq:defrho} shows that $\rho$ is the probability measure supported on the imaginary axis having the Cauchy distribution.  In this setting Theorem \ref{thm:mainsimple2} implies that, for any smooth and compactly supported function $\varphi:\mathbb{C} \to \mathbb{C}$, 
\[ \frac{1}{n} \sum_{i=1}^n \varphi(z_i^{(n)}) \longrightarrow \frac{1}{\pi} \int_{\mathbb{R}} \varphi(i y) \frac{1}{1 + y^2} dy \]
in probability as $n \to \infty$.  A numerical simulation of this example is presented in Figure \ref{fig:disks}.  
\end{example}

\begin{figure}[t] 
\centering
\includegraphics[width=0.5\textwidth]{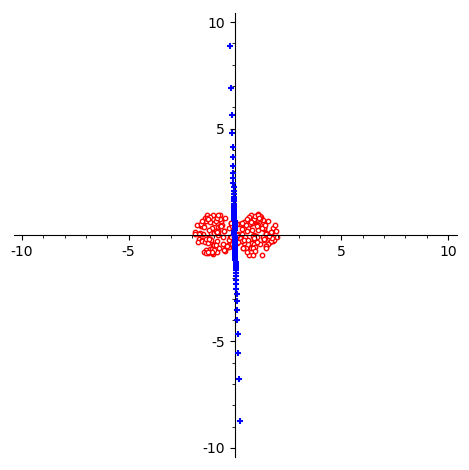}
\caption{A numerical simulation of Example \ref{example:disks}.  The red circles represent the roots of the individual polynomials (uniform on the unit disks centered at $1$ and $-1$, respectively), each having degree $n=100$, and the blue crosses are the zeros of the sum.}
\label{fig:disks}
\end{figure}

\begin{remark}
	In the above example one may replace the measures $\mu$ and $\nu$ by arbitrary radial probability measures centered at $1$ and $-1$ respectively and contained in their respective half planes. The resulting measure would still be the standard Cauchy measure on the imaginary axis.  A particular case is given by point masses at $1$ and $-1$. Then the corresponding sequences of polynomials are given by $\{(z-1)^n\}_{n\geq1}$ and $\{(z+1)^n\}_{n\geq1}$. The zero set of the sum $(z-1)^n+(z+1)^n$ is $\left\{-i\cot \left(\frac{(2k+1)\pi}{2n} \right) : 1\leq k \leq n \right\}$. Notice that the zeros are purely imaginary and the frequency on any interval is approximated by the standard Cauchy distribution on the imaginary axis. 
\end{remark}

\begin{remark}
Let $\mu:=\alpha \lambda_1 +(1-\alpha) \lambda_2 $ and $\nu:=\alpha \lambda_1 +(1-\alpha)\lambda_3$, where $0 \leq \alpha \leq 1$ and $\lambda_1, \lambda_2, \lambda_3$ are probability measures on $\mathbb{C}$ not supported on a circle. Then $\max \{U_\mu,U_\nu \}$ is computed to be $\alpha U_{\lambda_1}+(1-\alpha) \max \{ U_{\lambda_2},U_{\lambda_3} \}$.  Therefore, the resultant measure $\rho$ from Theorem \ref{thm:mainsimple2} is given by $\alpha \lambda_1 + (1-\alpha)\frac{1}{2\pi}\Delta \max \{ U_{\lambda_2},U_{\lambda_3} \}$ (see Theorem 3.7.4 in \cite{Ransford}). This result aligns with the case when the polynomials $p_n$ and $q_n$ have $\lfloor n\alpha \rfloor$ common factors, and the empirical distribution of common roots converges to $\lambda_1$.
\end{remark}

We generalize Theorem \ref{thm:mainsimple2} in two ways.  First, we consider the sum of more than two polynomials, and second, we consider more general probability measures.    

\begin{assumption} \label{assump:main}
Let $K \subset \mathbb{C}$ be compact.  We say \emph{$\mu, \nu \in \mathcal{P}(\mathbb{C})$ satisfy Assumption \ref{assump:main} on $K$} if there exists a finite constant $C > 0$ so that 
\begin{equation} \label{eq:assump}
	\lambda \left( \left\{z \in K : |U_\mu(z) - U_{\nu}(z) | \leq \frac{\log^2 n}{\sqrt{n}} \right\} \right) \leq \frac{C}{\log^3 n} 
\end{equation}
for all $n > C$.
\end{assumption}

Intuitively, \eqref{eq:assump} requires that $\mu$ and $\nu$ be distinct so that their corresponding logarithmic potentials differ on a large enough subset of $K$.  

\begin{theorem} \label{thm:mainsimple}
Let $m \geq 2$ be a fixed integer, and assume $\mu_1, \ldots, \mu_m \in \mathcal{P}(\mathbb{C})$.  Let $\varphi: \mathbb{C} \to \mathbb{C}$ be a smooth function with compact support (denoted $\supp(\varphi)$).  Assume one of the following conditions:
\begin{enumerate}
\item For each $1 \leq k \leq m-1$, the measure $\mu_k$ is not supported on a circle centered in $\supp(\varphi)$.
\item For each $1 \leq k < l \leq m$, the measures $\mu_k, \mu_l$ satisfy Assumption \ref{assump:main} on $\supp(\varphi)$.  
\end{enumerate}
Let $\{X_{i, k} : 1 \leq k \leq m, i \geq 1 \}$ be a collection of independent random variables so that $X_{i, k}$ has distribution $\mu_k$ for each $i \geq 1$.  For each $n \geq 1$, define the degree $n$ polynomials
\[ p_{n, k}(z) := \prod_{i=1}^n (z - X_{i, k} ), \qquad 1 \leq k \leq m. \]
Then there exists a (deterministic) probability measure $\rho$ on $\mathbb{C}$ so that
\[ \frac{1}{n} \sum_{i=1}^n \varphi(z_i^{(n)}) \longrightarrow \int_{\mathbb{C}} \varphi \d \rho \]
in probability as $n \to \infty$, where $z_1^{(n)}, \ldots, z_n^{(n)}$ are the zeros of the sum $\sum_{k=1}^m p_{n,k}$.  Here, $\rho$ depends only on $\mu_1, \ldots, \mu_m$ and is uniquely determined by the identity 
\[ \rho = \frac{1}{2 \pi} \Delta U, \]
where $U(z) := \max_{1 \leq k \leq m} U_{\mu_k}(z)$ for each $z \in \mathbb{C}$ and where the Laplacian is interpreted in the distributional sense (see Remark \ref{rem:distribution}).  
\end{theorem}

Theorem \ref{thm:mainsimple} can also be applied to measures supported on circles, as the following examples illustrate.  Let $B(z,r) := \{ w \in \mathbb{C} : |z - w| < r \}$ be the open disk of radius $r > 0$ centered at $z \in \mathbb{C}$.  

\begin{example} \label{example:circles}
Take $m=2$, and let  $\mu$ ($=\mu_1$) be the uniform probability measure on the circle of radius one centered at the origin and $\nu$ ($=\mu_2$) be the uniform probability measure on the circle of radius $r > 1$ centered at the origin.  Then Theorem \ref{thm:mainsimple} can be used to show that for any compactly supported and smooth function $\varphi: \mathbb{C} \to \mathbb{C}$, one has
\begin{equation} \label{eq:concircle}
	\frac{1}{n} \sum_{i=1}^n \varphi(z_i^{(n)}) \longrightarrow \int \varphi \d \nu 
\end{equation}
in probability as $n \to \infty$.  To see this, take $\eps := \frac{r - 1}{100}$ and note that if the compact set $K \subset \mathbb{C}$ is outside of $B(0, 1 + \eps)$, then $\mu$ is not supported on a circle centered in $K$.  Thus, Theorem \ref{thm:mainsimple} implies that \eqref{eq:concircle} holds for any $\varphi \in C^\infty_c(\mathbb{C})$  supported outside of $B(0, 1 + \eps)$.  Here, we have exploited the fact that $\max\{ U_{\mu}(z), U_{\nu}(z)\} = U_{\nu}(z)$ for all $z \in \mathbb{C}$ and $\frac{1}{2 \pi} \Delta U_{\nu} = \nu$ (see Theorem 3.7.4 in \cite{Ransford}).  In particular, by taking $\varphi$ to be an approximate indicator function, one deduces that
\[ \frac{ \# \left\{1 \leq i \leq n : z_i^{(n)} \in B(0, 1 + 3 \eps) \right\}}{n} \longrightarrow 0 \]
in probability as $n \to \infty$, where $\#S$ denotes the cardinality of the set $S$.  Thus, if $\varphi: \mathbb{C} \to \mathbb{C}$ is an arbitrary smooth and compactly supported function, then
\[ \frac{1}{n} \sum_{i=1}^n \varphi(z_i^{(n)}) - \frac{1}{n} \sum_{i=1}^n \varphi(z_i^{(n)}) \chi(z_i^{(n)}) \longrightarrow 0 \]
in probability, where $\chi: \mathbb{C} \to \mathbb{C}$ is a smooth approximation of an indicator function which takes the value $0$ inside $B(0, 1 + \eps)$ and which is one outside of $B(0, 1 + 2 \eps)$.  One can now utilize the previous convergence result for $\varphi$ supported outside of $B(0, 1 + \eps)$ to obtain the desired result.  A numerical simulation of this example is presented in Figure \ref{fig:circles}.  
\end{example}

\begin{figure}[t]
\centering
\includegraphics[width=0.5\textwidth]{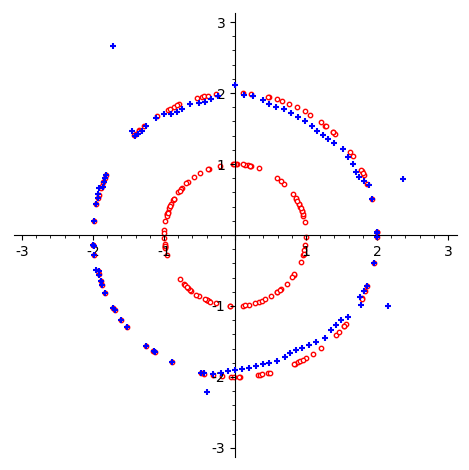}
\caption{A numerical simulation of Example \ref{example:circles} with $r=2$.  The red circles represent the roots of the individual polynomials (uniform on the circles centered at the origin with radii $1$ and $2$, respectively), each having degree $n=100$, and the blue crosses are the zeros of the sum.}
\label{fig:circles}
\end{figure}

\begin{remark}
A deterministic version of the above example can be seen by looking at the zeros of the polynomial sum $(z^n-1)+(z^n-2^n)$, which are given by $\left( \frac{2^n+1}{2} \right)^\frac{1}{n}w_k$ for $1 \leq k \leq n$, where $w_1,\dots,w_n$ are the $n$-th roots of unity. Therefore, the empirical distribution of zeros converges to the uniform distribution on the circle of radius $2$ centered at the origin. 
\end{remark}

\begin{example} \label{example:pointmass}
	Take $m=2$, and let $\mu$ ($=\mu_1$) and $\nu$ ($=\mu_2$) to be the uniform probability measures on the sets $\{\pm1\}$ and $\{\pm i\}$, respectively. Then $U_\mu$ and $U_\nu$ are computed to be $\frac{1}{2}\log|z^2-1|$ and $\frac{1}{2}\log|z^2+1|$. In this case, $\mu$ and $\nu$ are supported on circles, so Theorem \ref{thm:mainsimple2} does not apply.  However, $\mu$ and $\nu$ satisfy Assumption \ref{assump:main} on $\supp (\varphi)$ for any nonzero $\varphi \in C_c^\infty(\mathbb{C})$, hence Theorem \ref{thm:mainsimple} is applicable.  Let $z=x+iy$. Following the computations in Example \ref{example:disks}, we see the resulting measure $\rho$ is supported on the set $\Re(z^2)=0$ or, equivalently, on the lines $x+y=0$ and $x-y=0$. Moreover, it can be shown that the resulting measure is a mixture of distributions supported on these lines having density obtained by taking the square root of the absolute value of a Cauchy random variable after scaling by a factor of $2$. Notice, the tail behavior for the density of the resulting measure decays as $d^{-3}$ on either of the lines as $d$ tends to infinity, where $d$ is the distance from the origin. A numerical simulation of this example is presented in Figure \ref{fig:pointmass}.  Using a similar approach one may construct resulting measures $\rho$ whose densities have tails that decay as $d^{-2k-1}$ for any natural number $k$ as $d$ tends to infinity. 
\end{example}

\begin{figure}[t] 
\centering
\includegraphics[width=0.5\textwidth]{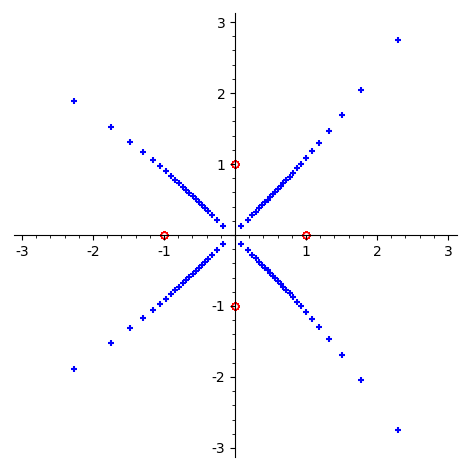}
\caption{A numerical simulation of Example \ref{example:pointmass}.  The red circles represent the roots of the individual polynomials (uniform on the sets $\{\pm 1\}$ and $\{\pm i\}$, respectively), each having degree $n=100$, and the blue crosses are the zeros of the sum.}
\label{fig:pointmass}
\end{figure}

\subsection{Comparison to free probability theory}

When $X$ and $Y$ are independent real-valued random variables with distributions $\mu_X$ and $\mu_Y$, the distribution of $X+Y$ is described by the convolution $\mu_X \ast \mu_Y$.  A natural way to describe the distribution $\mu_X \ast \mu_Y$ is to use characteristic functions since the characteristic function of $\mu_X \ast \mu_Y$ factors as the product of characteristic functions of $\mu_X$ and $\mu_Y$.

Free probability theory is concerned with non-commutative random variables.  The free analogue of the classical convolution is the free additive convolution of $\mu_X$ and $\mu_Y$, denoted $\mu_X \boxplus \mu_Y$, which describes the distribution of $X+Y$ when $X$ and $Y$ are freely independent non-commutative random variables.  The free additive convolution was introduced by Voiculescu in \cite{V} and later extended in \cite{BV,Mfree}.  Instead of characteristic functions, the free additive convolution $\mu_X \boxplus \mu_Y$ is naturally described using the $R$-transform.   Indeed, the $R$-transform of $\mu_X \boxplus \mu_Y$ can be written as the sum of $R$-transforms for $\mu_X$ and $\mu_Y$.  

%In classical probability theory, the convolution $\mu_X \ast \mu_Y$ is easily described by its characteristic function.  

The free additive convolution is especially useful in random matrix theory as it can be used to describe the limiting eigenvalue distribution of the sum of independent random matrices \cite{MS}.  The results in this paper focus on random polynomials rather than random matrices.  
In this way our results (Theorems \ref{thm:mainsimple2} and \ref{thm:mainsimple}) can be viewed as describing a version of the additive convolution from free probability theory for zeros of random polynomials.  The logarithmic potential appears to play the same role as the $R$-transform in free probability theory and the characteristic function in classical probability theory.  Our main results show that the pointwise maximum of the logarithmic potentials is used in the same way as the sum of $R$-transforms is used in free probability theory to describe the free additive convolution.

\subsection{Notation}

Throughout the paper, we use asymptotic notation (such as $O, o, \ll$) under the assumption that $n \to \infty$.  We write $X=O(Y)$, $Y=\Omega(X)$, $X \ll Y$, or $Y \gg X$ to denote the bound $|X| \leq CY$ for some constant $C > 0$ independent of $n$ and all $n > C$.  If the constant $C$ depends on a parameter, e.g., $C=C_k$, we indicate this with subscripts, e.g., $X=O_k(Y)$.  We allow the constant $C$ to depend on the measures in question (such as $\mu, \nu$ or $\mu_1, \ldots, \mu_m$) without denoting this dependence.  We use $X = o(Y)$ if $|X| \leq c_n Y$ for some $c_n$ that converges to zero as $n$ tends to infinity.  %We say an event $E$ (which depends on $n$) holds with \emph{overwhelming probability} if for every $\alpha > 0$, $\Prob(E) \geq 1 - O_\alpha(n^{-\alpha})$.  

%$\| \cdot \|_{\infty}$ denotes the $L^\infty$-norm.  %We use $\sqrt{-1}$ to denote the imaginary unit and reserve $i$ as an index.  We let $|S|$ denote the cardinality of the finite set $S$.
For $z \in \mathbb{C}$ and $r > 0$, we define
\[ B(z,r) := \{ w \in \mathbb{C} : |z - w| < r \} \]
to be the open disk of radius $r$ centered at $z$, and take 
$B(r) := B(0, r)$.    

Let $\mathcal{P}(\mathbb{C})$ be the set of probability measures on $\mathbb{C}$ with compact support.  For a measure $\mu \in \mathcal{P}(\mathbb{C})$, we let $\supp(\mu) \subset \mathbb{C}$ denote the support of $\mu$.  $\lambda$ denotes Lebesgue measure on $\mathbb{C}$.  Let $C^\infty_c(\mathbb{C})$ denote the set of all smooth functions $\varphi: \mathbb{C} \to \mathbb{C}$ with compact support, and let $\supp \varphi \subset \mathbb{C}$ denote the support of $\varphi$.  

\subsection{Overview of the proofs and outline of the paper}

The proofs of our main results are based on the following theorem.

\begin{theorem} \label{thm:main}
Let $m \geq 2$ be a fixed integer, and assume $\mu_1, \ldots, \mu_m \in \mathcal{P}(\mathbb{C})$.  Let $\{X_{i, k} : 1 \leq k \leq m, i \geq 1 \}$ be a collection of independent random variables so that $X_{i, k}$ has distribution $\mu_k$ for each $i \geq 1$.  For each $n \geq 1$, define the degree $n$ polynomials
\[ p_{n, k}(z) := \prod_{i=1}^n (z - X_{i, k} ), \qquad 1 \leq k \leq m. \]
Let $\varphi: \mathbb{C} \to \mathbb{C}$ be a smooth and compactly supported function, and assume that
\begin{equation} \label{eq:mainassump}
	\Prob \left( \bigcup_{k \neq l} \left\{ \frac{1}{2} \leq \left| \frac{p_{n,k}(Z)}{p_{n,l}(Z)} \right| \leq 2 \right\} \right) = O_{\varphi} \left( \frac{1}{ \log^3 n} \right), 
\end{equation}
where $Z$ is a random variable, uniformly distributed on the support of $\varphi$, independent of $X_{i, k}$, $1 \leq k \leq m$, $i \geq 1$.  Then
\[ \frac{1}{n} \sum_{i=1}^n \varphi(z_i^{(n)}) \longrightarrow \frac{1}{2 \pi} \int_{\mathbb{C}} \Delta \varphi(z) \left( \max_{1 \leq k \leq m} U_{\mu_k}(z) \right) \d \lambda(z) \]
in probability as $n \to \infty$, where $z_1^{(n)}, \ldots, z_n^{(n)}$ are the zeros of the sum $\sum_{k=1}^m p_{n,k}$.  
\end{theorem}

The proof of Theorem \ref{thm:main} is presented in Section \ref{sec:mainproof}.  We use Theorem \ref{thm:main} to prove our main results in Section \ref{sec:sideproofs}, where the main task is showing that condition \eqref{eq:mainassump} follows from the assumptions of Theorems \ref{thm:mainsimple2} and \ref{thm:mainsimple}.  In fact, the only place in the proof where we use that the measures $\mu_1, \ldots, \mu_{m-1}$ are not supported on circles (alternatively, the measures satisfy Assumption \ref{assump:main}) is in establishing condition \eqref{eq:mainassump} in Lemma \ref{lemma:circles} (alternatively, Lemma \ref{lemma:gap}).  To identify the limiting measure, we will utilize the following result.  

\begin{lemma} \label{lemma:genlap}
Let $m \geq 1$ be a fixed integer, and take $\mu_1, \ldots, \mu_m \in \mathcal{P}(\mathbb{C})$.  Then there exists a unique (deterministic) probability measure $\rho$ on $\mathbb{C}$ so that
\begin{equation} \label{eq:genlap}
	 \int_{\mathbb{C}} \varphi \d \rho = \frac{1}{2 \pi} \int_{\mathbb{C}} \Delta \varphi (z) \left( \max_{1 \leq k \leq m} U_{\mu_k}(z) \right) \d \lambda(z)  \quad \text{ for all } \varphi \in C^\infty_c (\mathbb{C}).
\end{equation}
In addition, \eqref{eq:genlap} uniquely defines $\rho$.  
\end{lemma}
\begin{proof}
The lemma follows from classical results on subharmonic functions and potential theory (see \cite{AG, Ransford}); we sketch the details below.  Since $U_{\mu_1}, \ldots, U_{\mu_m}$ are all subharmonic on $\mathbb{C}$ (see for instance Theorem 3.1.2 in \cite{Ransford}), it is easy to verify that the pointwise maximum is also subharmonic on $\mathbb{C}$.  By standard results for subharmonic functions (see Section 3.7 in \cite{Ransford} or Section 4.3 in \cite{AG}), there exists a unique Radon measure $\rho$ on $\mathbb{C}$ that satisfies \eqref{eq:genlap}.  The fact that \eqref{eq:genlap} uniquely defines $\rho$ follows from the Riesz representation theorem (see Section 3.7 in \cite{Ransford} or Section 4.3 in \cite{AG}).  To see that $\rho$ is a probability measure, we observe from Theorem 3.1.2 in \cite{Ransford} that there exists a constant $C > 0$ (depending on the measures $\mu_1, \ldots, \mu_m$) so that 
\begin{equation} \label{eq:logzasym}
	\left| \max_{1 \leq k \leq m} U_{\mu_k}(z) - \log |z| \right| \leq \frac{C}{|z|} 
\end{equation} 
for all $|z| \geq C$.  By taking $\varphi \in C^\infty_c (\mathbb{C})$ to be an approximate indicator function supported on a large disk centered at the origin, it follows from \eqref{eq:genlap} and the asymptotic behavior in \eqref{eq:logzasym} that $\rho$ is a probability measure.  
\end{proof}

\subsection*{Acknowledgements}

We are grateful to Noah Williams, Arjun Ramani, and Eojin Lee for discussions and assistance with the numerical simulations and figures in the paper.  We also thank the anonymous referee for useful feedback and suggestions.  T.R. Reddy is grateful to M. Krishnapur for the initial discussions on this problem and for the initial formulation of the solution. T.R. Reddy is grateful to  M. Sodin for sending his article \cite{Sodin} on value distributions of rational functions. T. R. Reddy is also grateful to S. O'Rourke for the hospitality during his visit to CU Boulder and NYUAD GPPO for supporting the travel.

\section{Proof of Theorem \ref{thm:main}} \label{sec:mainproof}

\subsection{Preliminary results}

We begin with some auxiliary results which we will need for the proof of Theorem \ref{thm:main}.  

\begin{lemma} \label{lemma:U2bound}
Suppose $\mu \in \mathcal{P}(\mathbb{C})$.  Then for each compact set $K \subset \mathbb{C}$, there exists a finite constant $C > 0$ (depending only on $\mu$ and $K$) so that 
\[ \int_{K} U^2_{\mu}(z) \d \lambda(z) \leq \int_{K} \int_{\mathbb{C}} \log^2 |z - w| \d \mu(w) \d \lambda(z) \leq C. \]
\end{lemma}
\begin{proof}
By Jensen's inequality and Fubini's theorem, 
\begin{align*}
	\int_{K} U^2_{\mu}(z) \d \lambda(z) &\leq \int_{K} \int_{\mathbb{C}} \log^2 |z - w| \d \mu(w) \d \lambda(z) \\
	&= \int_{\supp(\mu)}  \int_{K} \log^2 |z - w| \d \lambda(z) \d \mu(w),  
\end{align*}
where $\supp(\mu)$ denotes the support of $\mu$.  
Since the function
\begin{equation} \label{eq:funeq}
	w \mapsto \int_{K} \log^2 |z - w| \d \lambda(z) 
\end{equation} 
is continuous (see Lemma \ref{lemma:log}), it follows that the function in \eqref{eq:funeq} is bounded on compact sets.  The claim now follows as $\mu$ is compactly supported.  
\end{proof}

\begin{lemma} \label{lemma:pnqnbounds}
Under the assumptions of Theorem \ref{thm:main} and for any compact set $K \subset \mathbb{C}$, 
\begin{equation} \label{eq:pnqnbnd}
	\sup_{1 \leq k \leq m} \frac{1}{n^2} \int_{K} \log^2 |p_{n, k}(z)| \d \lambda(z) = O_{K}(1)
\end{equation}
and
\begin{equation} \label{eq:maxpnqn}
	\frac{1}{n^2} \int_{K} \log^2 \left( \max_{1 \leq k \leq m} |p_{n, k}(z)| \right) \d \lambda(z) = O_{K,m}(1)
\end{equation}
with probability $1$.  
\end{lemma}
\begin{proof}
We begin by establishing \eqref{eq:pnqnbnd}.  
By the Cauchy--Schwarz inequality,  
\[ \log^2 |p_{n,k}(z)| = \left( \sum_{i=1}^n \log |z - X_{i,k}| \right)^2 \leq n \sum_{i=1}^n \log^2 |z - X_{i,k}|, \]
and hence 
\begin{align*}
	\frac{1}{n^2} \int_{K} \log^2 |p_{n,k}(z)| \d \lambda(z) &\leq \frac{1}{n} \sum_{i=1}^n \int_{K} \log^2 |z - X_{i,k}| \d \lambda(z) \\
	&\leq \sup_{w \in \supp(\mu_k)} \int_{K} \log^2 |z - w| \d \lambda(z)   
\end{align*}
with probability $1$.  We conclude that
\[ \sup_{1 \leq k \leq m} \frac{1}{n^2} \int_{K} \log^2 |p_{n,k}(z)| \d \lambda(z) \leq \sup_{w \in K'} \int_{K} \log^2 |z - w| \d \lambda(z), \]
where $K' := \cup_{k=1}^m \supp(\mu_k)$.  
The claim now follows since the function in \eqref{eq:funeq} is bounded on compact sets (see Lemma \ref{lemma:log}) and $\mu_1, \ldots, \mu_m$ are compactly supported.  

The bound in \eqref{eq:maxpnqn} follows from \eqref{eq:pnqnbnd} and the trivial bound
\[ \log^2 \left(\max_{1 \leq k \leq m} a_k \right) \leq \sum_{k=1}^m \log^2 a_k, \]
valid for all $a_1, \ldots, a_m > 0$.  
\end{proof}

\begin{lemma} \label{lemma:tight}
Under the assumptions of Theorem \ref{thm:main}, there exists a finite constant $C > 0$ (depending only on $\mu_1, \ldots, \mu_m$ and $m$) so that, with probability $1$, all the roots of $\sum_{k=1}^m p_{n,k}$ are contained in the disk $B(Cn^{m-1})$.  
\end{lemma}
\begin{proof}
Since $\mu_1, \ldots, \mu_m$ have compact support, there exists a finite constant $M > 0$ so that $\supp(\mu_k) \in B(M)$ for $1 \leq k \leq m$.  This implies that $X_{i, k} \in B(M)$ with probability $1$ for all $i \geq 1$ and every $1 \leq k \leq m$.  The claim now follows from Lemma \ref{lemma:detsum}, which is a consequence of a deterministic bound due to Walsh \cite{Walsh}.  
\end{proof}

\begin{lemma} \label{lemma:logsumtight}
Under the assumptions of Theorem \ref{thm:main} and for any $M > 0$,  
\[ \frac{1}{n^2} \int_{B(M)} \log^2 \left| \sum_{k=1}^m p_{n,k}(z) \right| \d \lambda(z) = O_{M, m}(\log^2 n) \]
with probability $1$.  
\end{lemma}
\begin{proof}
Recall that $z_1^{(n)}, \ldots, z_n^{(n)}$ are the roots of $\sum_{k=1}^m p_{n,k}(z)$.  
Since 
\[ \sum_{k=1}^m p_{n,k}(z) = m \prod_{i=1}^n (z - z_i^{(n)}), \]
it follows from the Cauchy--Schwarz inequality that
\begin{align*}
	\log^2 \left| \sum_{k=1}^m p_{n,k}(z) \right| &= \left( \sum_{i=1}^n \log |z-z_i^{(n)}| + \log m \right)^2 \\
	&\leq (n+1) \left( \sum_{i=1}^n \log^2 |z - z_i^{(n)}| + \log^2 m \right). 
\end{align*}
We deduce that, with probability $1$, 
\begin{align*}
	\frac{1}{n^2} \int_{B(M)} \log^2 \left| \sum_{k=1}^m p_{n,k}(z) \right| \d \lambda(z) &\leq \frac{2}{n} \sum_{i=1}^n \int_{B(M)} \log^2 |z - z_i^{(n)}|  \d \lambda(z) + o_M(1) \\
	&\leq 2 \sup_{w \in B(Cn^{m-1})} \int_{B(M)} \log^2 |z - w| \d \lambda(z) + o_M(1),
\end{align*}
where in the last step we applied Lemma \ref{lemma:tight}.  The conclusion now follows from the fact that
\[ \sup_{w \in B(Cn^{m-1})} \int_{B(M)} \log^2 |z - w| \d \lambda(z) = O_{M, m}(\log^2 n), \]
which can be deduced by considering the cases $w \in B(2M)$ and $w \in B(Cn^{m-1}) \setminus B(2M)$ separately.  
\end{proof}

\subsection{A reduction}

In this section, we simplify the proof of Theorem \ref{thm:main} by exploiting the following result.  

\begin{lemma} \label{lemma:almostsure}
Let $m \geq 1$ be a fixed integer, and assume $\mu_1, \ldots, \mu_m \in \mathcal{P}(\mathbb{C})$.  Let $\{X_{i, k} : 1 \leq k \leq m, i \geq 1 \}$ be a collection of independent random variables so that $X_{i, k}$ has distribution $\mu_k$ for each $i \geq 1$.  For each $n \geq 1$, define the degree $n$ polynomials
\[ p_{n, k}(z) := \prod_{i=1}^n (z - X_{i, k} ), \qquad 1 \leq k \leq m. \]
If $\varphi: \mathbb{C} \to \mathbb{C}$ is a smooth and compactly supported function, then 
\[ \frac{1}{n} \int_{\mathbb{C}} \Delta \varphi(z) \log \left( \max_{1 \leq k \leq m} |p_{n,k}(z)| \right)  \d \lambda(z) \longrightarrow \int_{\mathbb{C}} \Delta \varphi(z) \left ( \max_{1 \leq k \leq m} U_{\mu_k}(z) \right) \d \lambda(z) \]
almost surely as $n \to \infty$.  
\end{lemma}

The proof of Lemma \ref{lemma:almostsure} is based on the following dominated convergence result due to Tao and Vu \cite{TVcirc}.  

\begin{lemma}[Dominated convergence; Lemma 3.1 from \cite{TVcirc}] \label{lemma:dominated}
Let $(X, \rho)$ be a finite measure space.  For integers $n \geq 1$, let $f_n: X \to \mathbb{R}$ be random functions which are jointly measurable with respect to $X$ and the underlying probability space.  Assume that:
\begin{enumerate}
\item (uniform integrability) there exists $\delta > 0$ such that $\int_X |f_n(x)|^{1 + \delta} \d \rho(x)$ is bounded in probability (resp., almost surely);
\item (pointwise convergence) for $\rho$-almost ever $x \in X$, $f_n(x)$ converges in probability (resp., almost surely) to zero.
\end{enumerate}
Then $\int_X f_n(x) \d \rho(x)$ converges in probability (resp., almost surely) to zero.  
\end{lemma}

\begin{proof}[Proof of Lemma \ref{lemma:almostsure}]
We will prove the lemma by applying Lemma \ref{lemma:dominated} with
\[ f_n(z) := \Delta \varphi(z) \left( \frac{1}{n} \log \left( \max_{1 \leq k \leq m } |p_{n,k}(z)| \right) - \max_{1 \leq k \leq m} U_{\mu_k}(z)  \right). \]
Assuming $\varphi$ is supported on $B(M)$ for $M > 0$ sufficiently large, it will suffice to consider the finite measure space $B(M)$ with the Lebesgue measure $\lambda$ when applying Lemma \ref{lemma:dominated}.  

First observe that since $\log | \cdot |$ is locally integrable on $\mathbb{C}$, Fubini's theorem implies that $U_{\mu_1}, \ldots, U_{\mu_m}$ are finite Lebesgue almost everywhere.  
Thus, by the law of large numbers, for Lebesgue almost every $z \in \mathbb{C}$, 
\[ \frac{1}{n} \log |p_{n,k}(z)| = \frac{1}{n} \sum_{i=1}^n \log |z - X_{i,k}| \longrightarrow U_{\mu_k}(z) \]
almost surely as $n \to \infty$ for $1 \leq k \leq m$.  It follows that, for Lebesgue almost every $z \in \mathbb{C}$, $f_n(z) \to 0$ almost surely as $n \to \infty$.  

In view of Lemma \ref{lemma:dominated} it remains to show that 
\[ \frac{1}{n^2} \int_{B(M)} | \Delta \varphi(z)|^2 \log^2 \left( \max_{1 \leq k \leq m} |p_{n,k}(z)| \right) \d \lambda(z) \]
and
\[ \int_{B(M)}  | \Delta \varphi(z)|^2 \left( \max_{1 \leq k \leq m} U^2_{\mu_k}(z) \right) \d \lambda(z) \]
are bounded almost surely.  Since $\varphi$ is smooth and compactly supported, $\Delta \varphi$ can be bounded in $L^\infty$-norm.  Therefore, it suffices to show that
\begin{equation} \label{eq:intbndpnqn}
	\frac{1}{n^2} \int_{B(M)} \log^2 \left( \max_{1 \leq k \leq m} |p_{n,k}(z)|\right) \d \lambda(z)
\end{equation}
and
\begin{equation} \label{eq:intbndU2}
	\int_{B(M)} \max_{1 \leq k \leq m } U^2_{\mu_k}(z) \d \lambda(z)
\end{equation}
are bounded almost surely.  A bound for \eqref{eq:intbndpnqn} follows from Lemma \ref{lemma:pnqnbounds}.  The bound for \eqref{eq:intbndU2} can be deduced from Lemma \ref{lemma:U2bound}.  
\end{proof}

In view of Lemma \ref{lemma:almostsure}, in order to prove Theorem \ref{thm:main}, it suffices to show that
\[  \frac{1}{n} \sum_{i=1}^n \varphi(z_i^{(n)}) - \frac{1}{2\pi n} \int_{\mathbb{C}} \Delta \varphi(z) \log \left( \max_{1 \leq k \leq m} |p_{n,k}(z)| \right) \d \lambda(z) \longrightarrow 0 \]
in probability as $n \to \infty$,  where $z_1^{(n)}, \ldots, z_n^{(n)}$ are the zeros of the sum $\sum_{k=1}^m p_{n,k}$.  We now exploit the following formula (see, for instance, Section 2.4.1 from \cite{HKYV}), that for any smooth and compactly supported function $\varphi:\mathbb{C} \to \mathbb{C}$, 
\[ \sum_{i=1}^n \varphi(z_i^{(n)}) = \frac{1}{2\pi} \int_{\mathbb{C}} \Delta \varphi(z) \log \left| \sum_{k=1}^m p_{n,k}(z) \right| \d \lambda(z). \]
Therefore, in order to complete the proof of Theorem \ref{thm:main}, it suffices to establish the following asymptotic result.  

\begin{lemma} \label{lemma:main}
Under the assumptions of Theorem \ref{thm:main}, 
\[  \frac{1}{n} \int_{\mathbb{C}} \Delta \varphi(z) \log \left| \sum_{k=1}^m p_{n,k}(z) \right| \d \lambda(z) - \frac{1}{n} \int_{\mathbb{C}} \Delta \varphi(z) \log \left( \max_{1 \leq k \leq m} |p_{n,k}(z)| \right) \d \lambda(z) \longrightarrow 0 \]
in probability as $n \to \infty$.  
\end{lemma}

The rest of this section is devoted to the proof of Lemma \ref{lemma:main}.  

\subsection{Proof of Lemma \ref{lemma:main}}

The following lemma will justify the comparison of the logarithm of the sum with the logarithm of the maximum in Lemma \ref{lemma:main}.

\begin{lemma} \label{lemma:sumtomax}
Let $Z$ be a random variable, uniformly distributed on the support of $\varphi$, independent of $X_{i,k}$, $1\leq k \leq m$, $i \geq 1$.  Then, under the assumptions of Theorem \ref{thm:main}, 
\[ \frac{1}{n} \log \left| \sum_{k=1}^m p_{n,k}(Z) \right| = \frac{1}{n} \log \left( \max_{1 \leq k \leq m} |p_{n,k}(Z)| \right) + O_m \left( \frac{1}{n} \right) \]
with probability $1 - O_{\varphi} ( (\log n)^{-3} )$.  
\end{lemma}

We delay the proof of Lemma \ref{lemma:sumtomax} until Section \ref{sec:proofsumtomax}. 
We now turn to the proof of Lemma \ref{lemma:main} which is based on the following Monte Carlo sampling result from \cite{TV}.  

\begin{lemma}[Monte Carlo sampling lemma; Lemma 6.1 from \cite{TV}] \label{lemma:monte}
Let $(X, \rho)$ be a probability space, and let $F: X \to \mathbb{C}$ be a square integrable function.  Let $l \geq 1$, let $Z_1, \ldots, Z_l$ be drawn independently at random from $X$ with distribution $\rho$, and let $S$ be the empirical average
\[ S := \frac{1}{l} ( F(Z_1) + \cdots + F(Z_l) ). \]
Then, for any $\delta > 0$, one has the bound
\[ \left| S - \int_X F \d \rho \right| \leq \frac{1}{\sqrt{l \delta} } \left(  \int_X \left| F - \int_X F \d \rho \right|^2 \d \rho \right)^{1/2} \]
with probability at least $1 - \delta$.   
\end{lemma}

\begin{proof}[Proof of Lemma \ref{lemma:main}]
Let
\[ F_n(z) := \frac{1}{n} \Delta \varphi(z) \left[ \log \left| \sum_{k=1}^m p_{n,k}(z) \right| - \log \left( \max_{1 \leq k \leq m} |p_{n,k}(z)| \right) \right]. \]
Take $K := \supp \varphi$, and choose $M > 0$ sufficiently large so that $K \subset B(M)$.  Our goal is to show that
\[ \int_{K} F_n(z) \d \lambda(z) \longrightarrow 0 \]
in probability as $n \to \infty$.  To do so we will apply the Monte Carlo sampling lemma (Lemma \ref{lemma:monte}) to $F_n$ with $\rho$ being the uniform probability distribution on $K$.  

We first note that
\begin{align*}
	\int_{\mathbb{C}} |F_n|^2 \d \rho &= \frac{1}{\lambda(K)} \int_{K} |F_n(z)|^2 \d \lambda(z) \\
	&\ll_K \frac{1}{n^2} \| \Delta \varphi \|^2_{\infty} \int_{B(M)} \left( \log^2 \left| \sum_{k=1}^m p_{n,k}(z) \right| + \log^2 \left( \max_{1 \leq k \leq m} |p_{n,k}(z)| \right) \right) \d \lambda(z),
\end{align*}
where $\| \Delta \varphi \|_{\infty}$ denotes the $L^\infty$-norm of $\Delta \varphi$.  
The integral can be bounded using Lemmas \ref{lemma:pnqnbounds} and \ref{lemma:logsumtight} to obtain
\begin{equation} \label{eq:varbound}
	\int_{\mathbb{C}} |F_n|^2 \d \rho = O_{m, \varphi}(\log^2 n) 
\end{equation}
with probability $1$.  
We are now in a position to apply the Monte Carlo sampling lemma.  Take $l := \lceil (\log n)^{5/2} \rceil$.  Let $Z_1, \ldots, Z_l$ be independent and identically distributed (iid) random variables with distribution $\rho$ (i.e., $Z_1, \ldots, Z_l$ are uniformly distributed on $K$), $\delta := (\log n)^{-1/4}$, and 
\[ S_n := \frac{1}{l} \sum_{j=1}^l F_n(Z_j). \]
Then Lemma \ref{lemma:monte} and \eqref{eq:varbound} imply that 
\[ \left| S_n - \frac{1}{\lambda(K)} \int_{K} F_n(z) \d \lambda(z) \right| = O_{m, \varphi}( (\log n)^{-1/8}) \]
with probability at least $1 - (\log n)^{-1/4}$.  We conclude that
\[ S_n - \frac{1}{\lambda(K)} \int_{K} F_n(z) \d \lambda(z) \longrightarrow 0 \]
in probability as $n \to \infty$.  Thus, it remains to show that $S_n \to 0$ in probability.  However this follows from Lemma \ref{lemma:sumtomax}.  Indeed, Lemma \ref{lemma:sumtomax} and the union bound give 
\[ \sup_{1 \leq j \leq l} |F_n(Z_j)| = O_{m, \varphi} \left( \frac{1}{n} \right) \]
with probability $1 - O_{\varphi}((\log n)^{-1/2})$.  This immediately implies that the empirical average $S_n$ converges to zero in probability as $n \to \infty$.  
\end{proof}

\subsection{Proof of Lemma \ref{lemma:sumtomax}} \label{sec:proofsumtomax}

We conclude this section with the proof of Lemma \ref{lemma:sumtomax}.

\begin{proof}[Proof of Lemma \ref{lemma:sumtomax}]
Define the event 
\[ \Omega := \bigcup_{k \neq l} \left\{ \frac{1}{2} \leq \left| \frac{ p_{n,k}(Z) }{ p_{n,l}(Z)} \right| \leq 2 \right\}, \]
which appears on the left-hand side of \eqref{eq:mainassump}.  Since $Z$ is continuously distributed, independent of $p_{n,k}$, with probability $1$, $p_{n,k}(Z) \neq 0$ for $1 \leq k \leq m$; this implies that $\Omega$ is well-defined except on events which hold with probability zero, which we safely ignore for the remainder of the proof.  

In view of assumption \eqref{eq:mainassump}, it suffices to prove the result on $\Omega^c$.  
Let $\pi: \{1, \ldots, m\} \to \{1, \ldots, m\}$ be a (random) bijection which orders the values $|p_{n,k}(Z)|$ in decreasing order: 
\[ |p_{n,\pi(1)}(Z)| \geq |p_{n, \pi(2)}(Z)| \geq \cdots \geq |p_{n, \pi(m)}(Z)|. \] 
($\pi$ will in general depend on $Z$ and $X_{i,k}$, $1 \leq k \leq m$, $1 \leq i \leq n$.) By construction, $\max_{1 \leq k \leq m} |p_{n,k}(Z)| = |p_{n, \pi(1)}(Z)|$.

On $\Omega^c$, it follows that
\[ |p_{n,\pi(j)}(Z)| \geq 2 |p_{n, \pi(j+1)}(Z)| \]
for $1 \leq j \leq m-1$, which implies that
\begin{equation} \label{eq:geobnd}
	\left| \frac{ p_{n, \pi(j)}(Z) }{ p_{n, \pi(1)}(Z) } \right| \leq \left( \frac{1}{2} \right)^{j-1}, \quad 1 \leq j \leq m.
\end{equation}
We conclude that, on $\Omega^c$, 
\begin{align*}
	\frac{1}{n} \log \left| \sum_{k=1}^m p_{n,k}(Z) \right| &= \frac{1}{n} \log |p_{n, \pi(1)}(Z)|  + \frac{1}{n} \log \left| 1 + \sum_{j=2}^{m} \frac{ p_{n,\pi(j)}(Z)}{ p_{n, \pi(1)}(Z)} \right| \\
	&= \frac{1}{n} \log \left( \max_{1 \leq k \leq m} |p_{n,k}(Z)| \right) + O_m \left( \frac{1}{n} \right), 
\end{align*}
where the error term is controlled using \eqref{eq:geobnd}.  The proof of the lemma is complete.   
\end{proof}

\section{Proof of main results} \label{sec:sideproofs}

In this section, we prove Theorems \ref{thm:mainsimple2} and \ref{thm:mainsimple} using Theorem \ref{thm:main}.  The key to both proofs is guaranteeing that the assumptions in those theorems imply condition \eqref{eq:mainassump}, which is the purpose of the next two lemmas.  

\begin{lemma} \label{lemma:circles}
Let $K \subset \mathbb{C}$ be a compact set with positive Lebesgue measure.  Let $\mu, \nu \in \mathcal{P}(\mathbb{C})$, and assume $\mu$ is not supported on a circle centered in $K$.  For each $n \geq 1$, define the degree $n$ polynomials 
\[ p_n(z) := \prod_{i=1}^n (z - X_i), \qquad q_n(z) := \prod_{i=1}^n (z - Y_i), \]
where $X_1, Y_1, X_2, Y_2, \ldots$ are independent random variables so that $X_i$ has distribution $\mu$ and $Y_i$ has distribution $\nu$ for each $i \geq 1$.  Then 
\[ \Prob \left( \frac{1}{2} \leq \left| \frac{p_{n}(Z)}{q_{n}(Z)} \right| \leq 2 \right) = O_{K} \left( \frac{1}{ \sqrt{n} } \right), \]
where $Z$ is a random variable, uniformly distributed on $K$, independent of $X_1$, $Y_1$, $X_2$, $Y_2, \ldots$.  
\end{lemma}

\begin{lemma} \label{lemma:gap}
Let $K \subset \mathbb{C}$ be a compact set with positive Lebesgue measure.  Let $\mu, \nu \in \mathcal{P}(\mathbb{C})$, and assume $\mu, \nu$ satisfy Assumption \ref{assump:main} on $K$.  For each $n \geq 1$, define the degree $n$ polynomials 
\[ p_n(z) := \prod_{i=1}^n (z - X_i), \qquad q_n(z) := \prod_{i=1}^n (z - Y_i), \]
where $X_1, Y_1, X_2, Y_2, \ldots$ are independent random variables so that $X_i$ has distribution $\mu$ and $Y_i$ has distribution $\nu$ for each $i \geq 1$.  Then 
\begin{equation} \label{eq:ratio}
	\Prob \left( \frac{1}{2} \leq \left| \frac{p_{n}(Z)}{q_{n}(Z)} \right| \leq 2 \right) = O_{K} \left( \frac{1}{ \log^3 n } \right), 
\end{equation}
where $Z$ is a random variable, uniformly distributed on $K$, independent of $X_1, Y_1, X_2, Y_2, \ldots$.  
\end{lemma}

Combining Theorem \ref{thm:main} and Lemma \ref{lemma:genlap} with Lemmas \ref{lemma:circles} and \ref{lemma:gap} completes the proof of Theorems \ref{thm:mainsimple2} and \ref{thm:mainsimple}.  The rest of this section is devoted to the proofs of Lemmas \ref{lemma:circles} and \ref{lemma:gap}.

\subsection{Proof of Lemma \ref{lemma:circles}}

We will need the following definitions and results in order to prove Lemma \ref{lemma:circles}.  

\begin{definition}[Small ball probabilities]
Let $\xi$ be a real-valued random variable.  The \emph{L\'{e}vy concentration function} of $\xi$ is defined as
\[ \Le(\xi,t) := \sup_{u \in \mathbb{R}} \Prob( | \xi  - u | \leq t ) \]
for all $t \geq 0$.  
\end{definition}

The L\'{e}vy concentration function bounds the \emph{small ball probabilities} for $\xi$, which are the probabilities that $\xi$ falls in an interval of length $2t$ on the real line.   We will also exploit the following inequality due to Rogozin \cite{R} (see also Theorem 2.15 in \cite{Petrov}).  

\begin{theorem}[Kolmogorov--Rogozin inequality, \cite{R}] \label{thm:KR}
Let $\xi_1, \ldots, \xi_n$ be independent real-valued random variables, and let $t_1, \ldots, t_n$ be some positive real numbers.  Then for any $t \geq \max_j t_j$ we have
\[ \mathcal{L} \left( \sum_{j=1}^n \xi_j, t \right) \leq C t \left( \sum_{j=1}^n (1 - \mathcal{L}(\xi_j, t_j)) t_j^2 \right)^{-1/2}, \]
where $C > 0$ is an absolute constant.  
\end{theorem}

The Kolmogorov--Rogozin inequality has previously been used in the study random polynomials, for example, in the work of Kabluchko and Zaporozhets \cite{KZ}.  
The last ingredient in the proof of Lemma \ref{lemma:circles} is the following.  

\begin{proposition} \label{prop:circles}
Let $K \subset \mathbb{C}$ be compact.  Assume $\mu \in \mathcal{P}(\mathbb{C})$, and let $X$ be a random variable with distribution $\mu$.  If $\mu$ is not supported on a circle centered in $K$, then there exist $p_0, \eps_0 \in (0,1/4)$ such that 
\begin{equation} \label{eq:levy}
	\sup_{r \geq 0} \Prob \left(  r e^{-\eps_0} \leq |X - z| \leq r e^{\eps_0} \right) \leq 1 - p_0 
\end{equation}
for all $z \in K$.  
\end{proposition}
\begin{proof}
Suppose $K$ is nonempty as the claim is trivial otherwise.  Since $K$ and the support of $\mu$ are bounded, there exists $M > 0$ so that
\[ \Prob( r e^{-\eps_0} \leq |X - z| \leq r e^{\eps_0} ) = 0 \]
for any $r \geq M$, every $z \in K$ and every choice of $\eps_0 \in (0,1/4)$.  It thus suffices to show that there exist $\eps_0, p_0 \in (0,1/4)$ such that
\[ \sup_{0 \leq r \leq M} \Prob( r e^{-\eps_0} \leq |X - z| \leq r e^{\eps_0} ) \leq 1 - p_0 \]
for all $z \in K$.  

Suppose to the contrary that there exists $r_n \in [0, M]$ and $z_n \in K$ so that
\[ \Prob( r_n e^{-1/n} \leq |X - z_n| \leq r_n e^{1/n} ) > 1 - 1/n \]
for all $n > 4$.  By compactness (and by passing to subsequences), we may assume that $r_n \to r \in [0,M]$ and $z_n \to z \in K$.  Let $\eps > 0$.  Then 
\begin{align*}
	|r_n e^{-1/n} - r| < \eps, \qquad |r_n e^{1/n} - r| < \eps, \qquad |z - z_n| < \eps
\end{align*}
for all sufficiently large $n$.  Thus, we find
\begin{align*}
	1 - 1/n &< \Prob ( r_n e^{-1/n} \leq |X - z_n| \leq r_n e^{1/n} ) \\
	&\leq \Prob(r - \eps \leq |X - z_n | \leq r + \eps ) \\
	&\leq \Prob( r - 2 \eps \leq |X - z| \leq r + 2\eps )
\end{align*}
since 
\[ \left| |X - z_n| - |X - z| \right| \leq |z - z_n| < \eps \]
for any value of $X$ by the reverse triangle inequality.  Taking $n \to \infty$, we conclude that
\[ \Prob( r - 2 \eps \leq |X - z| \leq r + 2 \eps ) = 1. \]
Since $\eps > 0$ was arbitrary, we take $\eps \to 0$ and use continuity of measure to deduce that $\Prob ( |X - z| = r) = 1$.  This implies that $\mu$ is supported on the circle $\{w \in \mathbb{C} : |w - z| = r\}$, a contradiction.  
\end{proof}

\begin{proof}[Proof of Lemma \ref{lemma:circles}]
Define the event
\[ \Omega := \left\{ \frac{1}{2} \leq \left| \frac{p_{n}(Z)}{q_{n}(Z)} \right| \leq 2 \right\}. \]
Since $Z$ is continuously distributed, independent of $p_n, q_n$, with probability $1$, $Z$ avoids the atoms of $\mu$ and $\nu$ (which are at most countable).  In other words, with probability $1$, $p_{n}(Z) \neq 0$ and $q_n(Z) \neq 0$; this implies that $\Omega$ is well-defined except on events which hold with probability zero, which we safely ignore for the remainder of the proof.  

By taking logarithms, we find
\[ \Prob(\Omega) = \Prob \left( \left| \log |p_n(Z)| - \log |q_n(Z)| \right| \leq \log (2) \right). \]
By conditioning on $Z$ to not be an atom of $\mu$ or $\nu$, it suffices to show that
\[ \sup_{z \in K'} \Prob \left( \left| \log |p_n(z)| - \log |q_n(z)| \right| \leq \log (2) \right) = O_K (n^{-1/2}), \]
where $K'$ is the set $K$ with the atoms of $\mu$ and $\nu$ removed.  
Lastly, by conditioning on $q_n$ (which is independent of $p_n$), it suffices to show that
\begin{equation} \label{eq:toshowlevy}
	\sup_{z \in K'} \Le( \log |p_n(z)|, \log(2) ) = O_K (n^{-1/2}). 
\end{equation}

To bound the left-hand side of \eqref{eq:toshowlevy} we will apply Theorem \ref{thm:KR} since $\log |p_n(z)| = \sum_{i=1}^n \log |X_i - z|$ is the sum of iid random variables.  To this end, we will need to show the existence of $\eps_0, p_o \in (0,1/4)$ so that
\[ \Le ( \log |X_1 - z|, \eps_0) \leq 1 - p_0 \]
for all $z \in K'$.  
Observe that for any $u \in \mathbb{R}$ and any $z \in K'$
\[ \Prob( | \log |X_1 - z| - u | \leq \eps_0) = \Prob( e^{u - \eps_0} \leq |X_1 - z| \leq e^{u + \eps_0} ). \]
It thus suffices to show that there exist $\eps_0, p_0 \in (0,1/4)$ so that 
\[ \sup_{r \geq 0} \Prob( r e^{-\eps_0} \leq |X_1 - z| \leq r e^{\eps_0} ) \leq 1 - p_0 \]
for all $z \in K$.  The existence of such values of $\eps_0$ and $p_0$ follows from Proposition \ref{prop:circles}.  Applying Theorem \ref{thm:KR}, we conclude that
\[ \sup_{z \in K'} \Le( \log |p_n(z)|, \log(2) ) \leq \frac{ C}{ \eps_0 \sqrt{n p_o} }, \]
where $C > 0$ is an absolute constant.  This establishes \eqref{eq:toshowlevy}, and the proof is complete.  
\end{proof}

\subsection{Proof of Lemma \ref{lemma:gap}}

The proof of Lemma \ref{lemma:gap} is based on the following concentration inequality.

\begin{lemma}[Concentration inequality] \label{lemma:concentration}
Under the assumptions of Lemma \ref{lemma:gap}, there exists a finite constant $C > 0$ (depending only on $K, \mu$ and $\nu$) such that
\begin{equation} \label{eq:conc1}
	\Prob \left( \left| \frac{1}{n} \log |p_n(Z)| - U_{\mu}(Z) \right| \geq t \right) \leq \frac{C}{n t^2} 
\end{equation}
and
\begin{equation} \label{eq:conc2}
	\Prob \left( \left| \frac{1}{n} \log |q_n(Z)| - U_{\nu}(Z) \right| \geq t \right) \leq \frac{C}{n t^2} 
\end{equation}
for every $t > 0$.  
\end{lemma}
\begin{proof}
We only prove \eqref{eq:conc1} as the proof of \eqref{eq:conc2} is identical.  The proof is based on a second moment bound.  Indeed, by Markov's inequality, 
\[  \Prob \left( \left| \frac{1}{n} \log |p_n(Z)| - U_{\mu}(Z) \right| \geq t \right) \leq \frac{1}{t^2} \E \left| \frac{1}{n} \log |p_n(Z)| - U_\mu(Z) \right|^2, \]
and so it suffices to show that
\begin{equation} \label{eq:concshow}
	\E \left| \frac{1}{n} \log |p_n(Z)| - U_\mu(Z) \right|^2 = \frac{1}{n^2} \E \left| \sum_{i=1}^n (\log |Z - X_i| - U_\mu(Z)) \right|^2 \leq \frac{C}{n}. 
\end{equation}

Observe that $\E \left[ \log |Z - X_i| \mid Z \right] = U_{\mu}(Z)$ almost surely and $\E \log |Z - X_i| = \E U_{\mu}(Z)$.  The last expectation is finite due to Lemma \ref{lemma:U2bound}.  It follows that the terms in the sum
\[ \sum_{i=1}^n (\log |Z - X_i| - U_\mu(Z))  \]
have mean zero.  The terms in this sum are also uncorrelated as can be seen by applying the law of iterated expectations: 
\begin{align*}
	\E &\left[ (\log |Z - X_i| - U_\mu(Z) )(\log |Z-X_j| - U_\mu(Z)) \right] \\
	&\qquad= \E \left[ \E\left[ (\log |Z - X_i| - U_\mu(Z) )(\log |Z-X_j| - U_\mu(Z)) \mid Z \right]  \right] \\
	&\qquad= \E \left[ \E\left[ \log |Z - X_i| - U_\mu(Z)  \mid Z \right] \E \left[ \log |Z-X_j| - U_\mu(Z) \mid Z \right] \right] \\
	&\qquad= 0
\end{align*}
for $i \neq j$.  
Therefore, we conclude that 
\begin{align*} 
	\E \left| \sum_{i=1}^n (\log |Z - X_i| - U_\mu(Z)) \right|^2 &\leq \sum_{i=1}^n \E \left| \log |Z - X_i| - U_{\mu}(Z) \right|^2 \\
	&\leq 2n \left( \E \log^2 |Z - X_1| + \E U^2_{\mu}(Z)  \right). 
\end{align*}
Both expectations on the right-hand side can be bounded by constants using Lemma \ref{lemma:U2bound} since
\[ \E \log^2 |Z - X_1| = \frac{1}{\lambda(K)} \int_{K} \int_{\mathbb{C}} \log^2 |z - w| \d \mu(w) \d \lambda(z) \]
and
\[  \E U^2_{\mu}(Z) = \frac{1}{\lambda(K)} \int_{K} U_{\mu}^2(z) \d \lambda(z). \] 
Combining the bounds above establishes \eqref{eq:concshow}, and the proof of the lemma is complete.  
\end{proof}

\begin{proof}[Proof of Lemma \ref{lemma:gap}]
In view of Lemma \ref{lemma:concentration}, the event 
\[ \Omega := \left\{ \left| \frac{1}{n} \log |p_n(Z)| - U_{\mu}(Z) \right| \leq \frac{\log^{3/2} n}{\sqrt{n}} \right\} \bigcap \left\{ \left| \frac{1}{n} \log |q_n(Z)| - U_{\nu}(Z) \right| \leq \frac{\log^{3/2} n}{\sqrt{n}} \right\}. \]
holds with probability at least $1 - O_{K}( (\log n)^{-3} )$.  Let $\mathcal{E}$ be the event on the left-hand side of \eqref{eq:ratio} whose probability we aim to bound.  
Then
\[ \Prob (\mathcal{E}) \leq \Prob(\mathcal{E} \cap \Omega) + \Prob(\Omega^c) = \Prob(\mathcal{E} \cap \Omega) + O_{K}((\log n)^{-3}), \]
and so it remains to show that $\Prob(\mathcal{E} \cap \Omega) = O_{K}((\log n)^{-3})$.  

On the event $\mathcal{E}$, it follows that
\[ \left| \frac{1}{n} \log |p_n(Z)| - \frac{1}{n} \log |q_n(Z)| \right| \leq \frac{\log 2}{n}. \]
Hence, on the event $\mathcal{E} \cap \Omega$, we obtain via the triangle inequality that 
\begin{align*}
	&\left| U_{\mu}(Z) - U_{\nu}(Z) \right| \\
	&\qquad\leq \left| U_{\mu}(Z) - \frac{1}{n} \log |p_n(Z)| \right| + \left| \frac{1}{n} \log |p_n(Z)| - \frac{1}{n} \log |q_n(Z)| \right| + \left| \frac{1}{n} \log |q_n(Z)| - U_{\nu}(Z) \right| \\
	&\qquad \leq 2 \frac{\log^{3/2} n}{\sqrt{n}} + \frac{\log 2}{n} \\
	&\qquad\leq \frac{\log^2 n}{\sqrt{n}}  
\end{align*}
for all sufficiently large $n$.  In other words, we have shown that
\[ \Prob( \mathcal{E} \cap \Omega) \leq \Prob \left( \left| U_{\mu}(Z) - U_{\nu}(Z) \right| \leq \frac{\log^2 n}{\sqrt{n}} \right) \]
for every sufficiently large $n$.  Since $Z$ is uniformly distributed on $K$
\begin{align*}
	\Prob \left( \left| U_{\mu}(Z) - U_{\nu}(Z) \right| \leq \frac{\log^2 n}{\sqrt{n}} \right) &= \frac{\lambda \left( \left\{ z \in K :  \left| U_{\mu}(z) - U_{\nu}(z) \right| \leq \frac{\log^2 n}{\sqrt{n}} \right\} \right) }{\lambda(K) } \\
	&= O_{K}( (\log n)^{-3}) 
\end{align*}
by Assumption \ref{assump:main}.  We conclude that $\Prob(\mathcal{E} \cap \Omega) = O_{K}((\log n)^{-3})$, and the proof is complete.  
\end{proof}

\appendix

\section{Deterministic tools}

This section contains some useful lemmas we require throughout the paper.  The first is based on a fairly standard argument.  Recall that $\lambda$ denotes Lebesgue measure on $\mathbb{C}$.  
\begin{lemma} \label{lemma:log}
Let $K \subset \mathbb{C}$ be a compact set, and let $p \geq 1$ be an integer.  Then the function $f: \mathbb{C} \to \mathbb{R}$ given by
\[ f(z) := \int_{K} \log^p |z - w| \d \lambda(w) \]
is continuous.
\end{lemma}
\begin{proof}
Fix $z \in \mathbb{C}$ and let $\{z_n\}$ be a sequence of complex numbers converging to $z$.  It suffices to show that $\lim_{n \to \infty} f(z_n) = f(z)$.  

Let $\delta > 0$.  We assume $n$ is sufficiently large so that $|z_n - z| < \delta/100$.  Then
\begin{align*}
	|f(z_n) - f(z)| &\leq \int_{B(z, \delta)} \left( \left | \log^p |z_n - w| \right| + \left| \log^p |z - w| \right| \right) \d \lambda(w) \\
	&\qquad\qquad + \int_{K \cap B(z, \delta)^c} \left| \log^p |z_n - w| - \log^p |z - w| \right| \d \lambda(w). 
\end{align*}
On $K \cap B(z, \delta)^c$ the functions $\log^p |z_n - w|$ and $\log^p |z - w|$ are uniformly bounded, and so by the dominated convergence theorem, 
\[ \lim_{n \to \infty} \int_{K \cap B(z, \delta)^c} \left| \log^p |z_n - w| - \log^p |z - w| \right| \d \lambda(w) = 0. \]
The term 
\[ \int_{B(z, \delta)} \left( \left | \log^p |z_n - w| \right| + \left| \log^p |z - w| \right| \right) \d \lambda(w) \]
can be made arbitrarily small (by taking $\delta$ sufficiently small) since the function $\log^p |\cdot|$ is locally integrable.  
\end{proof}

The following deterministic results control the magnitude of the zeros for the sum of polynomials and are based on a bound due to Walsh \cite{Walsh}.  Recall that $B(r)$ denotes the disk of radius $r > 0$ centered at the origin in the complex plane.  

\begin{lemma} \label{lemma:detsum2}
Let $p$ and $q$ be degree $n$ polynomials, and assume that $p$ is monic and $q$ has leading coefficient $\lambda \geq 1$.  If all the zeros of $p$ and $q$ are contained in $B(M)$ for some $M > 0$, then all the zeros of $p+q$ are contained in the disk centered at the origin of radius $2M / \sin(\pi/n)$.  
\end{lemma}
\begin{proof}
We write $p + q = p - (-\lambda) \frac{1}{\lambda} q$, and note that $p$ and $\frac{1}{\lambda} q$ are monic polynomials.  Moreover, all the zeros of $p$ and $\frac{1}{\lambda} q$ are contained in $B(M)$ by supposition.  
The claim now follows from a bound due to Walsh (see Theorem IV from \cite{Walsh} or Theorem (17,2a) on page 77 of \cite{M}), which (as a special case) guarantees that the roots of $p + q$ are contained in the disk centered at the origin with radius 
\[ \frac{ M + \lambda^{1/n} M}{\lambda^{1/n} \sin(\pi/n)} \leq \frac{2M}{\sin(\pi/n)}. \]
\end{proof}

\begin{lemma} \label{lemma:detsum}
Let $p_1, \ldots, p_m$ be degree $n$ monic polynomials.  If all the zeros of $p_1, \ldots, p_m$ are contained in $B(M)$ for some $M > 0$, then all the zeros of the sum $\sum_{k=1}^m p_k$ are contained in the disk centered at the origin of radius $\frac{2^{m-1} M}{ (\sin(\pi/n))^{m-1}}$.  
\end{lemma}
\begin{proof}
We proceed by induction on $m$.  If $m = 2$ the result follows from Lemma \ref{lemma:detsum2}.  Let $m \geq 3$, and suppose the zeros of $\sum_{k=1}^{m-1} p_k$ are contained in the disk centered at the origin of radius $\frac{ 2^{m-2} M}{ ( \sin(\pi/n))^{m-2} }$.  By assumption the zeros of $p_m$ are contained in 
\[ B(M) \subset B\left(\frac{ 2^{m-2} M}{ ( \sin(\pi/n))^{m-2} } \right). \]
Thus, by Lemma \ref{lemma:detsum2}, the zeros of $\sum_{k=1}^m p_k = \sum_{k=1}^{m-1} p_{k} + p_m$ are contained in the disk centered at the origin of radius $\frac{2^{m-1}M}{(\sin (\pi/n))^{m-1} }$, and the proof is complete.  
\end{proof}

\bibliographystyle{abbrv}
\bibliography{sums}

\end{document}